\documentclass[12pt]{article}
\usepackage{amsmath,amssymb, amsfonts, amsthm,amscd}
\usepackage[T2A]{fontenc}
\usepackage[cp1251]{inputenc}
\usepackage[english]{babel}
\usepackage{graphicx}

\theoremstyle{plain}
\newtheorem{theorem}{Theorem}
\newtheorem{lemma}{Lemma}
\newtheorem{proposition}{Proposition}
\newtheorem{corollary}{Corollary}

\theoremstyle{definition}

\begin{document}
\begin{center}
{\huge Clear elements and clear rings}
\end{center}
\vskip 0.1cm \centerline{{\Large Bohdan Zabavsky, \; Olha Domsha, \;  Oleh Romaniv}}

\vskip 0.3cm

\centerline{\footnotesize{Department of Mechanics and Mathematics, Ivan Franko National University;}}

\centerline{\footnotesize{Lviv Regional Institute for Public Administration of the National}}
\centerline{\footnotesize{Academy for Public Administration under the President of Ukraine;}}

\centerline{\footnotesize{Department of Mechanics and Mathematics, Ivan Franko National University;}}
 \centerline{\footnotesize{zabavskii@gmail.com, \; olya.domsha@ukr.net, \; oleh.romaniv@lnu.edu.ua}}
\vskip 0.5cm

\centerline{\footnotesize{May, 2020}}
\vskip 0.7cm

\footnotesize{\noindent\textbf{Abstract:} \textit{An element in a ring $R$ is called clear if it is the sum of unit-regular element and unit. An associative ring is clear if every its element is clear. 
In this paper we defined clear rings and extended many results to wider class. Finally, we  proved that a commutative B\'ezout domain is an elementary divisor ring if and only if every full matrix order 2 over it is nontrivial clear.} }

\vspace{1ex}
\footnotesize{\noindent\textbf{Key words and phrases:} \textit{B\'ezout domain; clean element; unit-regular element; full matrix; elementary divisor ring; clear element; clear ring}

}

\vspace{1ex}

\vspace{1,5truecm}

\normalsize

\section{Introduction}

The work in this paper is prompted by the looking at the two sets $ U(R)$ and $ U_{reg}(R)$ in ring $R$, which denote, respectively the unit group and the set of unit-regular elements in $R$. Certainly the units and unit-regular elements are key elements determining the structure of the ring.

The study of rings generated additively by their units started in 1953 when K.G.~Wolfson~\cite{intro1Wolf} and D.~Zelinsky \cite{intro2Ze} proved independently, that every linear transformation of a vector space $V$ over a division ring $D$ is the sum of two nonsingular linear transformations except when $\mathrm{dim}\,V=1$ and $D=\mathbb{Z}_2$.  This implies, that the ring of linear transformations $\mathrm{End}_D(V)$ is the sum of two units except for one obvious case, when $V$ is a one-dimensional space over $\mathbb{Z}_2$.

The ring in which every element is the sum of two units Vamos called 2-good rings~\cite{intro3Va}. The ring $R$ is called  von Neumann  regular if for any $a\in R$ there exists $x\in R$ such that $axa=a$. In 1958 Skornyakov \cite[Problem~20, p.167]{intro4Sco} formulated the question: "Is every element of  von Neumann regular ring (which does not have $\mathbb{Z}_2$ as a factor-ring) the sum of units?"

According to Ehrlich \cite{3Er} element $a\in R$ is \textit{unit-regular} if $a=aua$ for some unit $u\in  U(R)$. It is easy to see that $a$ is unit-regular if and only if $a$ is an idempotent times $a$ unit, if and only if $a$ is unit times idempotent.
Ring is called \textit{unit-regular} if every its element is such.

Note that if $a\in R$ is unit-regular element and 2 is unit element in $R,$ then $a$ can be written as $a=eu$, where $e$ is idempotent and $u$ unit in $R$. Now, since $2$ is unit in $R$, $1+e$ is unit with $(1+e)^{-1}=1-2^{-1}e$. This gives that $e=(e+1)-1$ is the sum of two units and hence $a$ is the sum of two units.

We say that ring $R$ is \textit{Henriksen elementary divisor ring} if for square matrix $A\in R^{n\times n}$ there exist invertible matrices $P,Q\in R^{n\times n}$ such that $PAQ$ is diagonal matrix \cite[p.10]{1Za}.

\begin{theorem}\cite[Theorem 11]{intro5Hen}
  Let $R$ be Henriksen elementary divisor ring. Then $R^{n\times n}$, $n>1$, are 2-good ring.
\end{theorem}

It is important to note that every unit-regular ring is Henriksen elementary divisor ring \cite{intro6Hen}.
Also note that unit and idempotent are unit-regular elements.

The next class of rings generated additively by their units and idempotents is clean ring. The notion of it was introduced in 1977 by Nicholson in \cite{4Ni}.
Thereafter such rings and their variations were intensively studied by many authors. Recall that element of ring $R$ is \textit{clean} if it is the sum of idempotent and unit of R. The ring R is \textit{clean} if every element of R is such \cite{4Ni}. Nicholson also showed that in ring where idempotents are central, any unit-regular element is clean \cite[Proposition~1.8]{4Ni}, but in noncommutative ring unit-regular element is not necessarily clean. For example, the unit-regular matrix $\left(\begin{smallmatrix} 12 & 5 \\ 0 & 0 \end{smallmatrix}\right)\in \mathbb{Z}^{2\times2}$ is not clean \cite[Example 3.12]{5KhLa}. At the same time note that Camillo and Yu showed that any unit-regular ring is clean \cite[Theorem~5]{intro7CaHu}.

 All commutative von Neumann regular rings, local rings, semi-perfect rings and ring $R^{n\times n}$ for any clean ring $R$ are the examples of clean rings. Clean rings are closely connected to some important notions of the ring theory. Such rings are of interest since they constitute a subclass of the so-called   exchange rings in the theory of noncommutative rings.

Module $M_R$ has the exchange property if for every module $A_R$ and any two decompositions
$$
A=M'\oplus N=\oplus_{i\in I} A_i
$$
with $M'\cong M$, there exist submodules $A_i'\subset A_i$ such that
$$
A=M'\oplus(\oplus_{i\in I} A_i).
$$
The module $M_R$ has finite exchange property if the above condition is satisfied and the index set $I$ is finite. Warfield \cite{intro11Warfi} called ring $R$ an \textit{exchange ring} if $R_R$ has the finite exchange property and he showed that this definition is left-right symmetric. Independently, Goodearl, Warfield \cite{intro12GoWar} and Nicholson \cite{4Ni} obtained the very useful characterization of $R$: it is an exchange ring if and only if for any $a\in R$ there exists idempotent $e\in R$ such that $e\in aR$ and $1-e\in (1-a)R$. In this case the element $a$ is called an exchange element.

In this paper we propose the concept of clear ring based on the concept of clear element. All, these results are mainly centered on the application around the classical and rather ancient problem of describing all the of an elementary divisor rings. An overview can be found  in~\cite{1Za}. In the case of commutative rings in ~\cite{2Za} the connection of elementary divisor ring with the existence of clean elements of these rings is proven. Therefore, the problem of studying the matrix rings over elementary divisor rings in this aspect is especial.

We will study matrix rings over elementary divisor ring and reveal connection to the theory of full matrices over certain classes of rings.

Our main results are the following.

\begin{theorem}\label{theor1.1}
Let $R$ be a commutative elementary divisor ring and $A$ is a full nonsingular matrix of $R^{2\times2}$. Then exist invertible matrices $P, Q\in GL_2(R)$ such that $PAQ$ is nontrivial clear element of $R^{2\times2}$.
\end{theorem}

\begin{theorem}\label{theor1.2}
Let $R$ be a commutative elementary divisor ring. Then every full nonsingular matrix  $A\in R^{2\times2}$ is nontrivial clear.
\end{theorem}

\begin{theorem}\label{theor1.3}
Let $R$ be a semi-simple commutative B\'ezout domain. The next statements are equivalent:
\begin{enumerate}
  \item $R$ is an elementary divisor ring;
  \item  any full nonsingular matrix  of $R^{2\times2}$ is  nontrivial clear.
\end{enumerate}
\end{theorem}

\section{Notations and preliminary results}

For ease of explanation, let's do some notations and recall some definitions.

Throughout the paper we suppose $R$ is an   associative ring with  non-zero unit and $U(R)$ its group of units.
The vector space of matrices over the ring $R$ of size $k\times l$ is denoted by $R^{k\times l}$ and group of units of the ring $R^{n\times n}$ by $GL_n(R)$. The Jacobson radical of $R$ is denoted by $J(R)$. If $J(R)=0$, we say that $R$ is a semi-simple.

A ring $R$ is called a \textit{right} (\textit{left}) \textit{B\'ezout rin}g if each   finitely generated right (left) ideal of $R$  is principal. A ring $R$ which  is simultaneously  right  and  left B\'ezout ring is called a {\it B\'ezout ring}.

Matrix $A$ over ring $R$ \textit{admits diagonal reduction} if there exist invertible matrices $P$ and $Q$ such that $PAQ$ is a diagonal matrix $(d_{ij})$ for which $d_{ii}$ is a total divisor $d_{i+1,i+1}$ (i.e. $Rd_{i+1,i+1} R\subseteq d_{ii}R\cap Rd_{ii}$) for each~$i$. Ring $R$ is called \textit{an elementary divisor ring} provided that every matrix over $R$ admits a diagonal reduction~\cite{1Za}.


We can define ranks of a  matrix $A$ over $R$  on their  rows  $\rho_r(A)$ and their columns $\rho_c(A)$, respectively (see \cite[p.\,247]{Cohn}).
The smallest  $m\in \mathbb{N}$ such that a matrix $A\in R^{k\times l}$ is a product of two matrices of size $k\times m$  and $m\times l$,  is called the {\it inner rank $\rho(A)$ } of $A$. Note that $\rho(A)\leq \min\{\rho_r(A), \rho_c(A)\}$ and the number   $\rho(A)$ does not change under elementary transformations. If $R$ is a right B\'ezout domain, then $\rho(A)=\rho_r(A)=\rho_c(A)$ for any $A$ over $R$. A  matrix $A\in R^{n\times n}$  is called {\it full} if $\rho(A)=n$ (see \cite[p.\,248]{Cohn}), i.e. $R^{n\times n}A R^{n\times n}=R^{n\times n}$.

In the sequel, we use the following result.

\begin{proposition}\label{prop2.1}
  The following statements hold:
  \begin{itemize}
    \item[i)] \cite[Corollary 3.7]{6KhLaNi} Let $R=S^{2\times2}$ where $S$ is any ring and $A=\left(\begin{smallmatrix} a & b \\ c & d                    \end{smallmatrix}\right)\in R$. If $b\in  U(S)$ or $c\in  U(S),$ then $A$ is clean element of $S^{2\times2}$.
    \item[ii)] \cite[Theorem 2]{7Du} A commutative B\'ezout domain $R$ is an elementary divisor ring if and only if for any nonsingular full matrix $A\in R^{2\times2}$ right (left) principal ideal $AR^{2\times2}$ ($R^{2\times2}A$) contains nontrivial idempotent.
    \item[iii)] \cite[Proposition 1.8]{4Ni} Every clean element of ring is an exchange element.
  \end{itemize}
\end{proposition}

\section{Clear element}

Different classes of rings have generalized clean rings by adding one or more adjectives such as "almost", "semi", "uniquely" and i.e. to "clean". We propose another generalization of clean elements namely clear elements. Describing full matrices over commutative elementary divisor ring based on this concept is able.

  Element $a$ of ring  $R$ is \textit{clear} if $a=r+u$ where $r$ is unit-regular element and $u\in U(R)$. The ring $R$ is \textit{clear} if every  its  element is such. An obvious example of clear elements are 0 and 1.

  Let $a\in R$ is clear element such that $a=r+u$, where  $r$ is unit-regular and $u\in U(R)$. If $r\ne 0$ and $r\notin U(R),$ we say that clear element $a$ is nontrivial. For example, if $R$ is 2-good ring (any element is the sum of two units), then any $a\in R$ has trivial representation $a=u+v$, where $u,v\in U(R)$, though it does not exclude nontrivial representation of elements of $R$.

Since idempotent of ring is obviously unit-regular element, we get the next result.

\begin{proposition}\label{newprop1}
  Any clean element is clear. Clean ring is clear.
\end{proposition}

Moreover, we have the following proposition.

\begin{proposition}\label{newprop2}
  Unit-regular element is clear.
\end{proposition}

\begin{proof}
  Let $aua=a$ for some $u\in U(R)$. Since $au$ is idempotent, so $au=1-e$ for some idempotent $e\in R$. Then $a=u^{-1}-eu^{-1}$. Since $-eu^{-1}$   is unit-regular,  $a$ is clear element.
\end{proof}

We point out that these proposition are not valid in reverse order. For example, matrix $\left(\begin{smallmatrix} 12 & 5 \\ 0 & 0 \end{smallmatrix}\right)\in \mathbb{Z}^{2\times2}$ is clear element, but as noted earlier, it is not clean. Or, to take another example, ring $\mathbb{Z}_4$ is clear ring, but it is not unit-regular.

The description of idempotent or unit-regular matrices in the rings of matrices over rings is enough actual task. For example:

\begin{proposition}\cite[Proposition 2.1]{new1Sa}\label{newprop3}
Let $R=S^{2\times2}$ where $S$ is any ring and let $A=\left(\begin{smallmatrix} a & 0 \\ b & 0 \end{smallmatrix}\right)\in R$. A matrix $A$ is unit-regular in $R$ if and only if exists an idempotent $e\in S$ and an unimodular column $\left(\begin{smallmatrix} a_1 \\ b_1 \end{smallmatrix}\right)\in S^{1\times 2}$ such that  $\left(\begin{smallmatrix} a\\ b \end{smallmatrix}\right)=\left(\begin{smallmatrix} a_1 \\ b_1 \end{smallmatrix}\right)e$.
\end{proposition}

Recall that column $\left(\begin{smallmatrix} a\\ b \end{smallmatrix}\right)$ is unimodular if $R a+R b=R$. Also note that this condition is obviously equivalent to $R a+R b=R e$ where $e^2=e$, $a=a_1e$, $b=b_1e$ and $R a_1+R b_1=R$.

\begin{proposition}\label{newprop4}
  Every homorphic image of a clear rings is clear.
\end{proposition}
\begin{proof}
Every unit-regular element is an idempotent times a unit element and since multiplication is preserved by every ring homomorphism. The homomorphic image of an idempotent and unit element are suitable idempotent and unit element of its ring.
\end{proof}

Since addition is also preserved by every ring homomorphism, the next result follows.

\begin{proposition}\label{newprop5}
Every direct product of clear rings is clear.
\end{proposition}
\begin{proof}
It is known that multiplication in direct product of rings is defined componentwise, so element in the direct product of rings is unit (resp. idempotent) of that ring if and only if the entry in each of its components is unit (resp. idempotent) of this ring. Since addition in a direct product of rings is also defined componentwise, the result follows from a simple computation.
\end{proof}

\begin{lemma}\label{newlemma6}
  Element $a\in R$ is clear if and only if $ua$ and $au$ are clean elements for some $u\in U(R)$.
\end{lemma}

\begin{proof}
    Let $a$ is clear element, i.e. $a=r+v$, where $r$ is unit-regular element and $v\in U(R)$. Since $rur=r$ for some $u\in U(R)$, then $ua=ur+uv$ where $(ur)^2=ur$ and $uv\in U(R)$, i.e. $ua$ is clean element. Similarly $au =ru+vu$ where  $(ru)^2=ru$ and $vu\in U(R)$.

    Let for $a\in R$ there exists unit element $u\in  U(R)$ such that $ua$ and $au$ are clean elements. Let $ua =e+v$ where $e^2=e$ and $v\in U(R)$. Then $a=u^{-1}e+u^{-1}v$. Since $(u^{-1}e)u(u^{-1}e)=u^{-1}e$, we have that $u^{-1}e$ is an unit-regular element and $u^{-1}v\in  U(R)$, i.e. $a$ is clear.
    The same is in the case $au$.
\end{proof}

In this context for describing clear rings let's give the following definition.

  Ring $R$ is said to have unit-regular stable range 1 if for any element $a,b\in R$ with $aR+bR=R$ there exists some unit-regular element $r$ of $R$ such that $a+br$ is an unit element of $R$.

\begin{lemma}\label{newtheor7}
    Ring of unit-regular stable range 1 is clear.
\end{lemma}

\begin{proof}
Let $R$ is ring of unit-regular stable range 1 and $a\in R$. Then $aR+(-1)R=R$ and we have that $a+(-1)r=u\in U(R)$ for some unit-regular element $r\in R$, i.e. $a=r+u$.
\end{proof}

For the following results recall that element $a\in R$ is \textit{2-clean} if $a=e+u+v$ for some idempotent $e$ and unit elements $u,v\in U(R)$ \cite{new3Gov}.

\begin{proposition}\label{newprop8}
For any commutative ring clear element is 2-clean. Any time clear ring is 2-clean ring.
\end{proposition}
\begin{proof}
Let $a\in R$ is clear element, i.e. $a=r+u$ for some unit-regular element $r$ and $u\in U(R)$. Since every unit-regular element of commutative ring is clean~\cite{new4WaChe} we have that $r=e+v$ for some idempotent $e\in R$ and $v\in U(R)$. Then $a=e+u+v$.
\end{proof}

In addition, we shall notice, that for any ring $2\times 2$ and $3\times 3$ matrices are 2-clean \cite[Lemma~3]{new4WaChe}.

Since local ring is clean by Proposition~\ref{newprop1} we also have the following result.

\begin{proposition}\label{newprop9}
Local ring is clear.
\end{proposition}

There are not nontrivial idempotents in the local ring. Let's describe clear ring $R$ which has not nontrivial idempotent (domain, for example). Recall that element $a\in R$ is \textit{2-good} if $a$ is the sum of two unit elements.

\begin{proposition}\label{newprop109}
The following are equivalent for any ring $R$:
\begin{enumerate}
  \item $R$ is a clear and has no nontrivial idempotents;
  \item for any element $a\in R:$
  \begin{enumerate}
  \item $a\in U(R)$ or
  \item $a$ is 2-good element.
\end{enumerate}

\end{enumerate}
\end{proposition}
\begin{proof}
Let $R$ be a clear and has not nontrivial idempotents.

Since 0 and 1 are the only ones idempotents in $R$, then 0 and $u\in U(R)$ are unit-regulars, i.e. all elements of $R$ are trivial clear elements, i.e. for any $a\in R$ we have that $a\in U(R)$ or $a$ is 2-good element.

The implication $(2)\to (1)$ is obviously.
\end{proof}

\section{Proofs}

We start our proofs with the following lemma.

\begin{lemma}\label{lem3.1}
  Let $R$ be a commutative elementary divisor ring, then for any full matrix $A\in R^{2\times2}$ there exist invertible matrices $P,Q\in GL_2(R)$ such that \[PAQ=\begin{pmatrix} 1 & 0 \\ 0 &  d  \end{pmatrix}\quad\text{ for some element }  d\in R.
  \]
\end{lemma}

\begin{proof}
As noted earlier, if $R$ is commutative elementary divisor ring, then for any full matrix $A\in R^{2\times2}$ there exist invertible matrices $P, T\in GL_2(R)$ such that $$PAT=\begin{pmatrix}  d_1 & 0 \\ 0 & d_2 \end{pmatrix}= D,$$ where $ d_1$ is a total divisor of $ d_2$. Since $R$ is commutative ring, then obviously $ d_1$ is divisor of $ d_2$. From
\[
R=R^{2\times2}\,A\,R^{2\times2}=R^{2\times2} D R^{2\times2}
\]
and from $ d_1$ is divisor $ d_2$ follows $ d_1\in  U(R)$, i.e. we can assume that $ d_1=1$.
\end{proof}

\begin{proof}[Proof of Theorem~\ref{theor1.1}]
Before starting Proof, note that if $A=\left(\begin{smallmatrix} 1 & 0 \\ 0 & d  \end{smallmatrix}\right)\in S^{2\times2}$ for any ring $S,$ then
  $$
   \begin{pmatrix} 1 & 0 \\ 0 &  d  \end{pmatrix}
   \begin{pmatrix} 0 &1 \\ 1 & 0  \end{pmatrix}=
   \begin{pmatrix} 0 & 1 \\  d &0  \end{pmatrix}\quad \text{ and }\quad  \begin{pmatrix} 0 & 1 \\  d &0  \end{pmatrix}=\begin{pmatrix} 0 & 0 \\  d+1 &1  \end{pmatrix}+\begin{pmatrix} 0 & 1 \\ -1 & -1  \end{pmatrix}.
  $$

 So (by Proposition~\ref{prop2.1}), we have that $\left(\begin{smallmatrix} 0 & 1 \\  d &0  \end{smallmatrix}\right)$ is clean matrix. Also note that $\left(\begin{smallmatrix} 0 & 1 \\ 1&0  \end{smallmatrix}\right)\in GL_2(R)$.

  According to our comments and Lemma~\ref{lem3.1} we are finishing the proving of Theorem~\ref{theor1.1}.
\end{proof}

\begin{proof}[Proof of Theorem~\ref{theor1.2}]
  By Theorem~\ref{theor1.1} we have that $PAQ=E+U$ where $E^2=E$ and invertible matrices $P,Q,U\in GL_2(R)$. So $A=P^{-1}EQ^{-1}+P^{-1}UQ^{-1}$ and $P^{-1}EQ^{-1}=(P^{-1}EQ^{-1})QP(P^{-1}EQ^{-1})$. Obviously $QP\in GL_2(R)$ and $P^{-1}EQ^{-1}$ is a unit-regular matrix and $P^{-1}UQ^{-1}\in GL_2(R)$. That is $A$ is clear matrix.
\end{proof}

As well known, every commutative principal ideal domain is elementary divisor domain \cite{new1Sa}. Consequently of Theorem~\ref{theor1.2} we have the following result.

\begin{corollary}\label{newcor1}
Let $R$ be a commutative principal ideal domain. Then any full matrix $A\in R^{2\times2}$ is clear.
\end{corollary}

Now let $R$ be a commutative B\'ezout domain in which any nonzero prime ideal contained in a unique maximal ideal, i.e. $R$ is $PM^*$ ring. By \cite[Theorem~1]{new5ZaGa} $R$ is an elementary divisor ring. So, in the same way from Theorem~\ref{theor1.2} we get the following result.

\begin{corollary}\label{newcor2}
Let $R$ be a commutative $PM^*$ B\'ezout domain. Then any full matrix $A\in R^{2\times2}$ is clear.
\end{corollary}

Note that elementary divisor ring is B\'ezout ring \cite[Theorem 1.2.7]{1Za}. Consequently of Theorem~\ref{theor1.2} and Proposition~\ref{prop2.1}(ii) we can proving Theorem~\ref{theor1.3}.

\begin{proof}[Proof of Theorem~\ref{theor1.3}]
By Theorem~\ref{theor1.2} we have implication $1)\rightarrow 2)$.

$2)\rightarrow 1)$. Let full matrix $A\in R^{2\times2}$ is clear. By Lemma~\ref{newlemma6} there exists invertible matrix $U\in GL_2(R)$ such that $UA$ and $AU$ are clean matricies. By Proposition~\ref{prop2.1}  $UA$ ($AU$) is an exchange element. By \cite[Proposition 1.1]{4Ni} a right (left) ideal $AUR^{2\times2}$ ($R^{2\times2}UA$) contains idempotent $E$ ($F$). Since $R$ is semi-simple and $J(R^{2\times2})=(J(R))^{2\times2}$ by proving of Proposition~1.9 \cite{4Ni}  idempotent $E$ $(F)$ is not trivial. By Proposition~\ref{prop2.1}(ii) we obtain that $R$ is an elementary divisor domain.
\end{proof}

\section{Some open questions}

\begin{description}
  \item[1.] Is the commutative clear ring a ring of unit-regular stable range 1?
  \item[2.] Is the notion of a ring of unit-regular stable range 1 a left-right symmetric?
\end{description}

\section*{Disclosure statement}

No potential conflict of interest was reported by the author(s).

\section*{ORCID}

\textit{Bohdan Zabavsky} http://orcid.org/0000-0002-2327-5379\\
\textit{Olha Domsha} http://orcid.org/0000-0003-1803-6055\\
\textit{Oleh Romaniv} http://orcid.org/0000-0002-8022-7859

\end{document}